\documentclass[11pt]{article}
\usepackage{latexsym, amsmath, amssymb}
\usepackage{graphicx,epsfig}
\usepackage{enumerate}

\newtheorem{thm}{Theorem}[section]

\newtheorem{prop}[thm]{Proposition}
\newtheorem{defn}[thm]{Definition}

\numberwithin{equation}{section}
\newcounter{saveeqn}

\newcommand{\cvd}{\hfill$\square$}
\newenvironment{proof}[1][Proof.]{\noindent\textbf{#1} }{\cvd}

\newcommand{\eqnref}[1]{(\ref {#1})}

\newcommand{\ga}{\gamma}
\newcommand{\Om}{\Omega}
\newcommand{\p}{\partial}

\newcommand{\Lcal}{\mathcal{L}}

\newcommand{\ve}{\varepsilon}
\newcommand{\vp}{\varphi}

\newcommand{\R}{\mathbb{R}}

\begin{document}

\title{Single-logarithmic stability \\ for the Calder\'{o}n problem with local data \thanks{Work supported in part by PRIN 20089PWTPS and by the JSPS International Training Program (ITP).}}
\author{Giovanni Alessandrini, \thanks{Dipartimento di Matematica e Geoscienze,
Universit\`a degli Studi di Trieste, via Valerio 12/1, 34127
Trieste, Italy. E-mail: \textsf{alessang@units.it}} \ Kyoungsun Kim \thanks{Department of Mathematics, Inha University, Incheon 402-751, Korea. E-mail: \textsf{kskim@inha.ac.kr}}}
\date{}
\maketitle

\begin{abstract} We prove an optimal stability estimate for Electrical Impedance Tomography with local data, in the case when the conductivity is precisely known on a neighborhood of the boundary.  The main novelty here is that we provide a rather general method which enables to obtain the
H\"older dependence of a global Dirichlet to Neumann
map from a local one on a larger domain when, in the layer between
the two boundaries, the coefficient is known.
\end{abstract}

\section{Introduction}

In this paper we consider the stability issue for Electrical Impedance Tomography (or, as is the same, the Calder\'{o}n problem) in the case of local boundary data. As a consequence of the recent results of uniqueness \cite{BU, KSU, IUY, I}, there has also been a number of results of stability \cite{HW1, HW2, F}, either directly associated to the conductivity equation setting, or associated to the Schr\"odinger equation setting. In general, the problem of the optimal rate of stability it is not yet settled.

It is expected however that, if the unknown coefficient (either the conductivity or the potential) is a-priori given on a neighborhood of the boundary then the rate of stability is of logarithmic type with a \emph{single} logarithm, which is optimal, in view of the well-known examples by Mandache \cite{M}. This is in fact the result of Fathallah \cite{F}, which along the lines of previous uniqueness results by Lassas, Cheney and Uhlmann \cite{LCU} and Ammari and Uhlmann \cite{AU}, proves a stability result with a single logarithm in the setting of the Schr\"odinger equation.

Here we consider the same situation directly for the conductivity equation, see Theorem \ref{main} in the next section, with the principal aim of providing a method of proof which might be adapted to other inverse boundary problems with local data and in which the unknown parameters, or objects (such as cavities, inclusions or cracks) are a-priori known to be located at a positive distance from the boundary.

The main idea in this method is that, assuming that the unknown part of a coefficient $\ga$ is concentrated in a subset $D\subset \subset \Om$, and if we fix a bigger domain $\widetilde D$ such that $D\subset \subset \widetilde D\subset \subset\Om$, then the full Dirichlet to Neumann map $\Lambda_{\ga}^{\p\widetilde D}$associated to $\widetilde D$ can be determined by the local Dirichlet to Neumann map $\Lambda_{\ga}^{\Sigma}$ associated to a portion $\Sigma$ of $\p\Om$. In fact such a dependence has a H\"older rate of stability. This is the essence of Theorem \ref{sub-main} below.

In Section \ref{prel} we formulate that main assumptions and state the main Theorem \ref{main}. Section \ref{proofmain} starts with some geometrical construction needed for the precise formulation of Theorem \ref{sub-main} which, in combination with the standard stability estimate with the full Dirichlet to Neumann map Theorem \ref{A}, enables a proof of Theorem \ref{main}. The final Section \ref{core} is devoted to the proof of Theorem \ref{sub-main}, this is mainly based on the use of singular solutions and on estimates of propagation of smallness for solutions of elliptic equations, some ideas are borrowed from previous work in \cite{ADC, AV, DiC}.

\section{Notation and main assumptions} \label{prel}
Let us introduce some notation and definitions.

Given $x\in \R^n$, $n\geq 3$, we shall denote $x=(x',x_n)$, where $x'=(x_1,\ldots,x_{n-1})\in\R^{n-1}$, $x_n\in\R$.
Given $x\in \R^n$, $r>0$, we shall use the following notation for balls and cylinders.
\begin{equation*}
   B_r(x)=\{y\in \R^n\ |\ |y-x|<r\}\ , \quad  B_r=B_r(0) \ ,
\end{equation*}
\begin{equation*}
   B'_r(x')=\{y'\in \R^{n-1}\ |\ |y'-x'|<r\}\ , \quad  B'_r=B'_r(0)\ ,
\end{equation*}
\begin{equation*}
   \Gamma_{a,b}(x)=\{y=(y',y_n)\in \R^n\ |\ |y'-x'|<a, |y_n-x_n|<b\}\ , \quad \Gamma_{a,b}=\Gamma_{a,b}(0) \ .
\end{equation*}
We shall denote by $\Omega$ a bounded
open connected subset of $\mathbb{R}^{n}$.
We shall assume that
its boundary is Lipschitz according to the following definition.
\begin{defn}
  \label{def:Lipschitz_boundary}
We say that the boundary of
$\Omega$ is of \emph{Lipschitz class} with
constants $\rho_{0}$, $M_{0}>0$, if, for any $P \in \partial\Omega$, there exists
a rigid transformation of coordinates under which $P=0$
and
\begin{equation}
   \label{bordo_lip1}
  \Omega \cap \Gamma_{\frac{\rho_{0}}{M_0},\rho_0}(P)=\{x=(x',x_n) \in \Gamma_{\frac{\rho_{0}}{M_0},\rho_0}\quad | \quad
x_{n}>Z(x')
  \} \ ,
\end{equation}
where $Z:B'_{\frac{\rho_{0}}{M_0}}\to\R$ is a Lipschitz function satisfying
\begin{equation}
   \label{bordo_lip2}
Z(0)=0\ ,
\end{equation}
\begin{equation}
   \label{bordo_lip3}
\|Z\|_{{C}^{0,1}\left(B'_{\frac{\rho_{0}}{M_0}}\right)} \leq M_{0}\rho_{0} \  .
\end{equation}
\end{defn}



We shall consider an open subset $\Sigma$ of $\partial\Omega$ which is accessible to measurements. We shall require a lower bound on its smallness as follows.
\begin{defn}
  \label{def:size_Sigma}
We shall say that $\Sigma$ has \emph{size at least} $d_0$, $0<d_0\leq\rho_0$, if there
exists at least one point $P\in\Sigma$ such that
\begin{equation}
   \label{r(P)_big}
  \mathrm{dist}(P,\partial\Omega\setminus \Sigma)\geq d_0 \ .
\end{equation}
\end{defn}
Consider a conductivity coefficient $\ga$ defined in
 $\Om$ and let us assume it is a bounded measurable function which
satisfies the ellipticity condition
\begin{equation}\label{ellipticity}
 \lambda < \ga(x) < \lambda^{-1} \quad \mbox{ for all  } x\in
\Om \ ,\end{equation} for a given positive constant $\lambda$.

\begin{defn} \label{def: local spaces}
We define the space of localized Dirichlet data as follows
\begin{equation*}
H_{co}^{1/2}(\Sigma)=\{\vp\in H^{1/2}(\p\Om)\ | \
\mathrm{supp}\vp\subset\Sigma\}\end{equation*}
and we denote by $H_{co}^{-1/2}(\Sigma)$
its topological dual space. We denote with $\langle \cdot, \cdot\rangle$ the dual pairing between these two spaces, based on the standard $L^2(\p\Om)$ inner product.
We shall denote by
$\|\cdot\|_{\Lcal(H_{co}^{1/2}(\Sigma),H_{co}^{-1/2}(\Sigma))}$ the operator norm on the space of bounded linear operators from $H_{co}^{1/2}(\Sigma)$ into $H_{co}^{-1/2}(\Sigma)$.
\end{defn}

Now, for each $\vp \in H_{co}^{1/2}(\Sigma)$, consider  the weak solution  $u\in H^1(\Om)$ to the Dirichlet problem

\begin{equation}\label{DP}\begin{cases}
\mbox{div}(\ga\nabla u)=0\quad \mbox{in  } \Om  \ ,\\
u|_{\p\Om}= \vp \quad \mbox{on  } \p\Om \ .\end{cases}\end{equation}

We introduce the local Dirichlet to Neumann map $\Lambda_{\ga}^{\Sigma}$
as the map which associates to the Dirichlet data $\vp \in H_{co}^{1/2}(\Sigma)$ the boundary co-normal derivative $\ga\frac{\p u}{\p\nu}\Big|_{\Sigma}$, where
$\nu$ is the outward unit normal to $\p\Om$. More precisely, we introduce the following definition.

\begin{defn}\label{def: local map}
 The map
 \begin{equation*}
\Lambda_{\ga}^{\Sigma}\ :\ H_{co}^{1/2}(\Sigma)  \rightarrow
H_{co}^{-1/2}(\Sigma)
\end{equation*}
is the operator characterized by
\begin{equation}\label{bilinloc}
\langle \Lambda_{\ga}^{\Sigma}\vp, \psi\rangle = \int_{\Om}\ga\nabla
u\cdot\nabla v,\quad \mbox{for every  }\vp, \psi\in
H_{co}^{1/2}(\Sigma)\ ,\end{equation}
where $u$ is the solution to the Dirichlet
problem \eqnref{DP} and $v$ is any function in $H^1(\Om)$ such that
$v|_{\p\Om}=\psi$.
\end{defn}
We shall consider an open subset $D$ of $\Omega$ whose boundary is also Lipschitz with constants $\rho_0, M_0$ and which is at a given positive distance from $\partial \Omega$, namely we assume
\begin{equation} \label{eq: distD}
\mathrm{dist}(D, \partial \Omega) \geq \rho_0 \ .
\end{equation}
On the unknown conductivity $\gamma$ we shall assume the following a-priori regularity bound
\begin{equation}\label{bound}
\|\ga\|_{W^{2,\infty}(\Om)} \leq E\ .\end{equation}
and also that it is precisely known outside $D$. That is, we assume that we are given a reference conductivity $\gamma_0$ which satisfies \eqref{ellipticity} and \eqref{bound} and the unknown $\gamma$ satisfies
\begin{equation}\label{gammaknown}
\gamma = \gamma_0 \mbox{ in  } \Omega \setminus \overline{D}\ .
\end{equation}
\begin{thm}\label{main}
Let $\Omega, \Sigma$ and $D$ satisfy the above stated assumptions. If $\gamma_1, \gamma_2$ satisfy \eqref{ellipticity}, \eqref{bound} and \eqref{gammaknown}, then we have
\begin{equation*}
\|\ga_1-\ga_2\|_{L^{\infty}(\Om)}\leq \omega \left(
\|\Lambda_{\ga_1}^{\Sigma}-\Lambda_{\ga_2}^{\Sigma}\|_{\Lcal(H_{co}^{1/2}(\Sigma),H_{co}^{-1/2}(\Sigma))}\right)
\ ,\end{equation*}
where $\omega(t)$ is an increasing function of $t\geq 0$ such that
\begin{equation}\label{log}
\omega(t) \leq C |\log t|^{-\delta} \mbox{for every  } 0<t<e^{-1} \ ,
\end{equation}
here $C>0$ only depends on the a-priori data $\lambda, E, \rho_0, M_0, d_0, \mathrm{diam}(\Omega)$ and on $n$, whereas $\delta\in(0,1)$ only depends on $n$.\end{thm}

\section{Proof of Theorem \ref{main}} \label{proofmain}
Before formulating the main new tool (Theorem \ref{sub-main} below) that we shall use for the proof of Theorem \ref{main}, we need to introduce some geometrical constructions. We use notation and some results described in \cite{IP}.

First we introduce an augmented domain $\widetilde\Om$ by attaching to $\Omega$ an open set $\mathcal{A}$ lying in its exterior and whose boundary intersects $\partial \Omega$ on an open portion $\Sigma_0 \subset\subset \Sigma$. We refer to \cite[Section 6]{IP} for details.  In particular, we can choose $\mathcal{A}$ in such a way that, setting $\widetilde\Om = \Om \cup \Sigma_0 \cup \mathcal{A}$,  the following properties hold.

There exist $\rho_1, M_1>0$, only depending on $\rho_0, M_0, d_0$, such that
\begin{enumerate}[(i)]
\item $\widetilde\Om$ is open, connected and has Lipschitz boundary with constants $\rho_1, M_1>0$,
\item There exists $Q\in \mathcal{A}$ such that
$$
B_{2\rho_1}(Q) \subset \mathcal{A}  \ .
$$
\end{enumerate}
Next, if we denote, for any open set $E\in \mathbb{R}^n$ and $h>0$,
$$
E_h= \left\{x\in E | \mathrm{dist}(x, \partial E)>h \right\}
$$
we observe that there exists $h_0>0$ only depending on $\rho_0, M_0, d_0$ such that $\widetilde\Om_{h}$ is connected for every $h\leq h_0$, see for instance \cite[Proposition 5.5]{IP}. Note that $B_{\rho_1}(Q) \subset \widetilde\Om_{h}$ if $h\leq \rho_1$.

We introduce two domains $D^{\prime}, \widetilde{D}$ nested as follows
$$
D\subset \subset D^{\prime} \subset \subset \widetilde{D} \subset \subset \Om \ .
$$
Such domains can be chosen in such a way that for suitable $\rho_2, M_2>0$, only depending on $\rho_0, M_0, d_0$,  we have
\begin{enumerate}[(i)]
\item $\widetilde{\Om} \setminus \overline{D^{\prime}}$ and $\widetilde{\Om} \setminus \overline{\widetilde{D}}$ are connected,
\item $D^{\prime}, \widetilde{D}$ have $C^2$ boundaries, satisfying a Lipschitz condition with constants $\rho_2, M_2$,
\item the boundaries of $D, D^{\prime},  \widetilde{D},  \widetilde\Om_{\rho_2}$ have mutual distance greater than $\rho_2$.
\end{enumerate}
Let us incidentally note that the set $\widetilde{D}$ shall be used right away in the following statement, Theorem \ref{sub-main}, whereas the introduction of $D^{\prime}$ shall be justified during the proof of the same Theorem \ref{sub-main}.

Now we introduce the \emph{usual} Dirichlet to Neumann map $\Lambda_{\ga}^{\p\widetilde D}$ associated to the domain
$\widetilde{D}$. Namely, for a conductivity coefficient $\ga$ satisfying \eqref{ellipticity}, we consider, for any $\eta\in H^{1/2}(\p\widetilde D)$,
the solution $v\in H^{1}(\widetilde D)$ to the Dirichlet problem
\begin{equation}\label{v}\begin{cases}
\mbox{div}(\ga\nabla v)=0 \quad \mbox{in  }\widetilde D\ ,\\
v=\eta \quad \mbox{on  } \p\widetilde D\ ,\end{cases}\end{equation}
and we define
\begin{equation}\label{inner}
\Lambda_{\ga}^{\p\widetilde D}(\eta)=\ga\frac{\p
v}{\p\nu}\Big|_{\p\widetilde D}\ ,\end{equation} where
$\nu$ is the outward unit normal to $\p\widetilde D$. Again,
$\Lambda_{\ga}^{\p\widetilde D}$ is identified through the bilinear form characterization
\begin{equation}\label{bilinear}
\langle \Lambda_{\ga}^{\p\widetilde D}\eta, \xi\rangle =\int_{\widetilde
D}\ga\nabla v \cdot \nabla w\ , \quad \mbox{for every  }
\eta\ , \xi \in H^{1/2}(\p\widetilde D)\ ,\end{equation}
where $v$ is the solution to \eqnref{v} and $w$
is any function in $H^1(\widetilde D)$ such that $w|_{\p\widetilde D}=\xi$.

\begin{thm}\label{sub-main}
Let $\Omega, \Sigma$, $D$ and $\widetilde{D}$ satisfy the above stated assumptions. If $\gamma_1, \gamma_2$ satisfy \eqref{ellipticity}, \eqref{bound} and \eqref{gammaknown}, then we have
\begin{equation}\label{est-main}
\|\Lambda_{\ga_1}^{\p\widetilde D}-\Lambda_{\ga_2}^{\p\widetilde
D}\|_{\Lcal(H^{1/2}(\p\widetilde D), H^{-1/2}(\p\widetilde D))}\leq
C\|\Lambda_{\ga_1}^{\Sigma}-\Lambda_{\ga_2}^{\Sigma}\|^{\beta}_{\Lcal(H_{co}^{1/2}(\Sigma),
H_{co}^{-1/2}(\Sigma))}\ ,\end{equation}
where $ C>0, \beta\in(0,1)$ only depend on $\lambda, E, \rho_0, M_0, d_0, \mathrm{diam}(\Omega)$ and on $n$.\end{thm}
The proof of this theorem is the content of the next section. The other main ingredient for the proof of Theorem \ref{main} is the following known stability result for the Calder\'{o}n problem with full boundary data, see \cite{asing} and also, for details, \cite{A}.
\begin{thm}\label{A} Let $\widetilde{D}$ be as above.
Suppose that $\ga_1, \ga_2$ satisfy \eqref{ellipticity} and
\eqref{bound}. We have
the following stability estimate
\begin{equation}\label{AA}
\|\ga_1-\ga_2\|_{L^{\infty}(\widetilde D)}\leq
\omega(\|\Lambda_{\ga_1}^{\p\widetilde D}-\Lambda_{\ga_2}^{\p\widetilde
D}\|_{\Lcal(H^{1/2}(\p\widetilde D),H^{-1/2}(\p\tilde
D))})\ ,\end{equation}
where   $\omega$ is a logarithmic modulus of continuity satisfying
\begin{equation*}
\omega(t) \leq C |\log t|^{-\delta} \mbox{for every  } 0<t<e^{-1} \ ,
\end{equation*}
and $C>0$ only depends on the a-priori data $\lambda, E, \rho_0, M_0, d_0, \mathrm{diam}(\Omega)$ and on $n$, whereas $\delta\in(0,1)$ only depends on $n$.\end{thm}
Hence, assuming Theorem \ref{sub-main} proven, we can conclude as follows.

\begin{proof}[Proof of Theorem \ref{main}.]
Let us denote $\|\Lambda_{\ga_1}^{\Sigma}-\Lambda_{\ga_2}^{\Sigma}\|=\ve$. By \eqref{est-main}, we have
\begin{equation*}
\|\Lambda_{\ga_1}^{\p\widetilde D}-\Lambda_{\ga_2}^{\p\widetilde
D}\|_{\Lcal(H^{1/2}(\p\widetilde D), H^{-1/2}(\p\widetilde D))}\leq
C\varepsilon^{\beta}\end{equation*}
where, without loss of generality we may assume $C\geq 1$. If, on one hand, we have $C\varepsilon^{\beta}< e^{-1}$ then by \eqref{AA}
\begin{align*}
\|\ga_1-\ga_2\|_{L^{\infty}(\Om)}=\|\ga_1-\ga_2\|_{L^{\infty}(\widetilde D)}\leq \\ \leq C |\log (C \ve^{\beta})|^{-\delta} \leq C\left(\frac{\log Ce}{\beta}\right)^{\delta} |\log \ve|^{-\delta} \ .
\end{align*}
On the other hand, if $C\varepsilon^{\beta}\geq e^{-1}$ then, trivially,
\begin{align*}
\|\ga_1-\ga_2\|_{L^{\infty}(\Om)}\leq \lambda^{-1}\leq \\ \leq \lambda^{-1}\left(\frac{\log Ce}{\beta}\right)^{\delta} |\log \ve|^{-\delta}
\end{align*}
and the thesis follows.
\end{proof}

\section{Proof of Theorem \ref{sub-main}} \label{core}
Let us begin by observing that the reference conductivity $\ga_0$ can be extended  to $\widetilde \Om$ in such a way that the ellipticity condition \eqref{ellipticity} continues to hold in all of $\widetilde \Om$ and that the following Lipschitz bound holds
\begin{equation}\label{boundLip}
\|\ga_0\|_{W^{1,\infty}(\widetilde \Om)} \leq E_1
\end{equation}
where $E_1$ only depends on $E, \rho_0, M_0, d_0$. The same extension can be used for any $\ga$ satisfying \eqref{gammaknown}, and from now on we shall replace this assumption with the following one
\begin{equation}\label{gammaknownExt}
\gamma = \gamma_0 \mbox{ in  } \widetilde\Omega \setminus \overline{D} \ .
\end{equation}


Let $\ga_1, \ga_2$ be the two conductivities  appearing in the statement of Theorem \ref{sub-main},
 let us introduce the Green's function $G_i(x,y)$ , $i=1,2$, for the operator
for $\mathrm{div} (\ga_i\nabla)$ in the domain $\widetilde\Om$, that is, for any $y\in \widetilde\Om$,
$G_i(\cdot,y)$ is defined as the distributional solution to
\begin{equation*}\begin{cases}
\mathrm{div}_x(\ga_i(\cdot)\nabla_x G_i(\cdot,y))= - \delta(\cdot-y) \quad
\mbox{ in  }\widetilde\Om\ ,\\
G_i(\cdot,y)=0 \quad \mbox{ on }
\p\widetilde\Om\ .\end{cases}\end{equation*}
As is well-known, $G_i(x,y)$ is symmetric, it has a singularity on the diagonal $\left\{x=y\right\}$ of the order of $|x-y|^{2-n}$, and, away from the diagonal, it is $C^{1,\alpha}$-regular in each of the two variables $x,y$, moreover also the mixed derivatives $\nabla_x \nabla_y G_i(x,y)$ exist and are locally H\"older continuous away from the diagonal, see for instance \cite[Theorem 8.32]{GT}. In particular we shall make use of the following energy bound.

\begin{prop}\label{AV} For every $y\in\widetilde\Om$ and every $r>0$ we
have
\begin{equation}\label{est2}
\int_{\widetilde\Om\setminus B_r(y)}|\nabla_x G_i(x,y)|^2dx \leq
Cr^{2-n}\ ,
\end{equation}
where $C>0$ only depends on $\lambda$ and
$n$.\end{prop}
\begin{proof} A proof can be easily obtained through a Caccioppoli type inequality and the well-known pointwise upper bound of the Green's function \cite{LSW}, details can be found in \cite[Proposition 3.1]{AV}.
\end{proof}

Let us fix $\eta_i\in
C^{1,\alpha}(\p\widetilde D)$, $i=1,2$ for some $\alpha\in (0,1)$ and consider $v_i$ be solutions to the Dirichlet problems
\begin{equation*}\begin{cases}
\mathrm{div}(\ga_i\nabla v_i)=0 \quad \mbox{in  } \widetilde D \ ,\\
v_i=\eta_i \quad \mbox{on  } \p\widetilde D\ .\end{cases}\end{equation*}

By Green's formula, for every
$x\in \widetilde{D}$, we obtain
\begin{align*}
v_i(x) & = \int_{\p\widetilde D}\Big[\ga_i(z)\frac{\p
v_i}{\p\nu}(z)G_i(x,z)-v_i(z)\ga_i(z)\frac{\p
G_i}{\p\nu_z}(x,z)\Big]d\sigma_z\\
& = \int_{\p\widetilde D}\ga_0(z)\Big[\frac{\p
v_i}{\p\nu}(z)G_i(x,z)-v_i(z)\frac{\p
G_i}{\p\nu_z}(x,z)\Big]d\sigma_z\ .\end{align*}
Note that, by the regularity assumptions on the conductivities and on $\p\widetilde D$, $v_1, v_2$ are differentiable, in the classical sense, up to the boundary of $\widetilde{D}$ and that differentiation under the integral is elementarily allowed in the above formulas. Therefore, by Fubini's theorem,
for every $x\in D$
\begin{align} \label{innerp}
\nabla v_1(x)\ & \cdot \nabla v_2(x)= \nonumber \\
 = & \int_{\p\widetilde D \times \p\widetilde D}\ga_0(z)\ga_0(w)\frac{\p
v_1}{\p\nu}(z)\frac{\p v_2}{\p\nu}(w)\nabla_x G_1(x,z)\cdot \nabla_x
 G_2(x,w)d\sigma_z \times \sigma_w \nonumber \\
& -\int_{\p\widetilde D\times \p\widetilde D}\ga_0(z)\ga_0(w)\frac{\p
v_1}{\p\nu}(z)v_2(w)\nabla_x G_1(x,z) \cdot \nabla_x \frac{\p
G_2}{\p\nu_w}(x,w)d\sigma_z \times \sigma_w \nonumber \\
& -\int_{\p\widetilde D\times \p\widetilde D} \ga_0(z)\ga_0(w)v_1(z)\frac{\p
v_2}{\p\nu}(w)\nabla_x \frac{\p G_1}{\p\nu_z}(x,z)\cdot \nabla_x
G_2(x,w)d\sigma_z \times \sigma_w \nonumber \\
& + \int_{\p\widetilde D\times\p\widetilde D} \ga_0(z)\ga_0(w) v_1(z)
v_2(w)\nabla_x\frac{\p G_1}{\p\nu_z}(x,z)\cdot \nabla_x\frac{\p
G_2}{\p\nu_w}(x,w)d\sigma_z \times \sigma_w \end{align}
where $\sigma$ denotes the surface measure on $\p\widetilde D$.
For any $z,w \in \widetilde{\Om} \setminus \overline{D}$ let us
define
\begin{equation}\label{Sdef}
S(z,w)=\int_D (\ga_1(x)-\ga_2(x))\nabla_x G_1(x,z)\cdot \nabla_x
G_2(x,w)dx\ .\end{equation}
Note that, by Proposition \ref{AV}, such integral is well defined and in fact, if $z,w$ are such that $\mathrm{dist}(z, \p D), \mathrm{dist}(w, \p D) \geq R>0$ then we have
\begin{equation} \label{boundS}
|S(z,w)|\leq C R^{2-n}
\end{equation}
where $C>0$ only depends on $\lambda$ and $n$.
Then, by a well-known identity stemming from \eqref{bilinear}, we have
\begin{equation}
\langle (\Lambda_{\ga_1}^{\p\widetilde D}  -\Lambda_{\ga_2}^{\p\widetilde
D})\eta_1,\eta_2\rangle =   \int_{\widetilde D}(\ga_1-\ga_2)\nabla v_1\cdot \nabla v_2
dx
\end{equation}
consequently, by \eqref{gammaknownExt}, \eqref{innerp}, \eqref{Sdef}
\begin{align}\label{eqlambda}
 \langle (\Lambda_{\ga_1}^{\p\widetilde D}  -\Lambda_{\ga_2}^{\p\widetilde
D})\eta_1,\eta_2\rangle = \int_D(\ga_1-\ga_2)\nabla v_1 \cdot \nabla v_2 dx =\nonumber \\& = I_1 - I_2 - I_3 +I_4
\end{align}
where we denote
\begin{align}
& I_1= \int_{\p\widetilde D \times \p\widetilde D}\ga_0(z)\ga_0(w)\frac{\p
v_1}{\p\nu}(z)\frac{\p v_2}{\p\nu}(w)S(z,w)d\sigma_z \times \sigma_w\ , \label{eq1}\\
& I_2= \int_{\p\widetilde D \times \p\widetilde D}\ga_0(z)\ga_0(w)\frac{\p
v_1}{\p\nu}(z)v_2(w)\frac{\p}{\p\nu_w}S(z,w)d\sigma_z \times \sigma_w\ , \label{eq2}\\
& I_3 = \int_{\p\widetilde D \times \p\widetilde D} \ga_0(z)\ga_0(w)v_1(z)\frac{\p
v_2}{\p\nu}(w)\frac{\p}{\p\nu_z}S(z,w)d\sigma_z \times \sigma_w\ ,\label{eq3}\\
& I_4 = \int_{\p\widetilde D \times \p\widetilde D} \ga_0(z)\ga_0(w) v_1(z)
v_2(w)\frac{\p}{\p\nu_z}\frac{\p}{\p\nu_w}S(z,w)d\sigma_z \times \sigma_w\ .\label{eq4}\end{align}
For $z,w\in \mathcal{A}$,  the Green's functions
$G_1(\cdot,z) , G_2(\cdot,w) $ have no singularity in $\Om$ and also $G_1(\cdot,z)|_{\p\Om}, G_2(\cdot,w)|_{\p\Om}\in H_{co}^{1/2}(\Sigma)$.
More specifically, if $z,w \in B_{\rho_1}(Q)$, then, by \eqref{est2},
\begin{equation*}
\|G_1(\cdot,z)\|_{H_{co}^{1/2}(\Sigma)}, \|G_2(\cdot,w)\|_{H_{co}^{1/2}(\Sigma)} \leq C
\end{equation*}
where $C$ only depends on $\lambda, \rho_0, M_0, d_0$  and $n$.
Thus, recalling \eqref{bilinloc}, for $z,w\in B_{\rho_1}(Q)$, $S(z,w)$ can be rewritten as follows
\begin{align*}
S(z,w)&=\int_D (\ga_1(x)-\ga_2(x))\nabla_x G_1(x,z)\cdot \nabla_x
G_2(x,w)dx=\\
& = \int_{\Om} (\ga_1(x)-\ga_2(x))\nabla_x G_1(x,z)\cdot \nabla_x
G_2(x,w)dx=\\
& = \langle
(\Lambda_{\ga_1}^{\Sigma}-\Lambda_{\ga_2}^{\Sigma})G_1(\cdot,z),
G_2(\cdot,w)\rangle\ .\end{align*}
Hence,
\begin{equation}\label{errorboundS}
|S(z,w)|  \leq C \varepsilon \mbox{ for every } z,w\in B_{\rho_1}(Q)
\end{equation}
where we denote
\begin{equation} \label{epsilon}
\varepsilon =\|\Lambda_{\ga_1}^{\Sigma}-\Lambda_{\ga_2}^{\Sigma}\|_{\Lcal(H_{co}^{1/2}(\Sigma), H_{co}^{-1/2}(\Sigma))}  \ .
\end{equation}
On the other hand, recalling \eqref{boundS}, the following bound holds
\begin{equation} \label{boundDprime}
|S(z,w)|\leq C  \mbox{ for every } z, w \in \widetilde{\Om} \setminus \overline{D^{\prime}}
\end{equation}
where $C>0$ only depends on $\lambda, \rho_0, M_0, d_0$ and $n$.
Moreover, we have the following.
\begin{prop} \label{equationsforS}
For every $w\in \widetilde{\Om} \setminus \overline{D}$, the functions $S(\cdot,w), \frac{\p}{\p w_i}S(\cdot,w), i=1,\ldots,n$ are weak solutions to the elliptic equation
\begin{equation}\label{eqforS}
\mathrm{div}(\gamma_0 \nabla v) = 0 \mbox{ in } \widetilde{\Om} \setminus \overline{D} \ ,
\end{equation}
likewise, for every $z\in \widetilde{\Om} \setminus \overline{D}$, the functions $S(z,\cdot), \frac{\p}{\p z_i}S(z,\cdot), i=1,\ldots,n$ are weak solutions to the same equation.
\end{prop}
\begin{proof}
It suffices to verify the weak formulation of \eqref{eqforS} with an arbitrary test function $\psi \in C^{\infty}_0(\widetilde{\Om} \setminus \overline{D})$. This follows in a straightforward fashion, by repeated use of differentiation under the integral and of Fubini's theorem. Note that use is made of the assumption \eqref{gammaknownExt}.
\end{proof}

Using the fact that $S$ solves an elliptic equation in each variable, we can combine the smallness estimate \eqref{errorboundS} with the global bound \eqref{boundDprime}, so to estimate the smallness of $S$ and its derivatives on $\p\widetilde D \times \p\widetilde D$. This task can be achieved by an estimate of propagation of smallness, for a general discussion on this concept we refer to \cite{IP}. Let us fix $h_1 <\rho_2/2$ only depending on $\rho_0, M_0$ such that $(\widetilde{\Om} \setminus \overline{D^{\prime}})_{h_1}$ is connected.

\begin{prop}\label{propsmall}
If $v$ is a weak solution to
\eqref{eqforS} then
\begin{equation}\label{eq:propsmall}
\|v\|_{L^2((\widetilde{\Om} \setminus \overline{D^{\prime}})_{h_1})}\leq C \|v\|^{\eta}_{L^2(B_{\rho_1}(Q))}\|v\|^{1-\eta}_{L^2(\widetilde{\Om} \setminus \overline{D^{\prime}})}
\end{equation}
where $C>0, \eta\in(0,1)$ only depend on the a-priori data $\lambda, E, \rho_0, M_0, d_0, \mathrm{diam}(\Omega)$ and on $n$.
\end{prop}
\begin{proof} We refer to \cite[Theorem 5.1]{IP}, of which this proposition is just a special case.
\end{proof}

Applying Proposition \ref{propsmall} to $v=S(\cdot,w)$, for any $w\in B_{\rho_1}(Q)$ we obtain
\begin{equation}\label{errorpropzS}
\|S(\cdot,w)\|_{L^2((\widetilde{\Om} \setminus \overline{D^{\prime}})_{h_1})}  \leq C \varepsilon^{\eta} \mbox{ for every } w\in B_{\rho_1}(Q)
\end{equation}
and by a further application of Proposition \ref{propsmall}, with respect to the $w$ variable, we have
\begin{equation}\label{errorpropzwS}
\|S(\cdot,\cdot)\|_{L^2((\widetilde{\Om} \setminus \overline{D^{\prime}})_{h_1}\times (\widetilde{\Om} \setminus \overline{D^{\prime}})_{h_1})}  \leq C \varepsilon^{\eta^2} \ .
\end{equation}
Using the elliptic equation for $S(\cdot,w)$ and the fact that $\p\widetilde D$ is contained in $(\widetilde{\Om} \setminus \overline{D^{\prime}})_{h_1}$  and at a distance greater than $\rho_2/2$ from $\partial (\widetilde{\Om} \setminus \overline{D^{\prime}})_{h_1}$, by standard interior regularity estimates \cite[Theorems 8.24, 8.32]{GT} we deduce
\begin{equation}\label{errorpropgradzS}
\|S(z,\cdot)\|_{L^2((\widetilde{\Om} \setminus \overline{D^{\prime}})_{h_1})} + \|\nabla_zS(z,\cdot)\|_{L^2((\widetilde{\Om} \setminus \overline{D^{\prime}})_{h_1})}\leq C \varepsilon^{\eta^2} \mbox{ for every } z \in \p\widetilde D \ ,
\end{equation}
using now the equation for $S(z,\cdot)$ and its first order $z$-derivatives, the interior regularity estimates give
\begin{align}\label{errorpropgradzwS}
\|S\|_{L^{\infty}(\p\widetilde D\times \p\widetilde D)} + \|\nabla_zS\|_{L^{\infty}(\p\widetilde D\times \p\widetilde D)}+\nonumber &\\ \|\nabla_wS\|_{L^{\infty}(\p\widetilde D\times \p\widetilde D)}+ \|\nabla_z\nabla_wS\|_{L^{\infty}(\p\widetilde D\times \p\widetilde D)}\leq \nonumber &\\&\leq C \varepsilon^{\beta}
\end{align}
where we denote $\beta=\eta^2$. Let us now combine the above bounds with \eqref{eqlambda}--\eqref{eq4}. For instance, we write
\begin{equation*}
I_1= \langle \Lambda_{\ga_1}^{\p\widetilde D} \eta_1, g\rangle
\end{equation*}
where
\begin{equation*}
g(z)=\langle \Lambda_{\ga_2}^{\p\widetilde D} \eta_2, S(z,\cdot)\rangle \ ,
\end{equation*}
hence, with some crude majorization,
\begin{align*}
|I_1| \leq C\| \eta_1\|_{H^{1/2}(\p\widetilde D)}\| g\|_{H^{1/2}(\p\widetilde D)}& \leq \\
\leq & C\| \eta_1\|_{H^{1/2}(\p\widetilde D)}\left( \| g\|_{L^{\infty}(\p\widetilde D)}+  \| \nabla g\|_{L^{\infty}(\p\widetilde D)} \right) \ .
\end{align*}
By the same reasoning and by \eqref{errorpropgradzwS} we also have
\begin{equation*}
\left( \| g\|_{L^{\infty}(\p\widetilde D)}+  \| \nabla g\|_{L^{\infty}(\p\widetilde D)} \right) \leq C\|\eta_2\|_{H^{1/2}(\p\widetilde D)} \ve^{\beta} \ ,
\end{equation*}
and consequently
\begin{equation*}
|I_1| \leq  C\ve^{\beta}\|\eta_1\|_{H^{1/2}(\p\widetilde
D)}\|\eta_2\|_{H^{1/2}(\p\widetilde D)} \ .
\end{equation*}
Using a similar approach for the terms in \eqref{eq2}--\eqref{eq4} we arrive at
\begin{equation}
|\langle(\Lambda_{\ga_1}^{\p\widetilde D}-\Lambda_{\ga_2}^{\p\widetilde
D})\eta_1,\eta_2\rangle| \leq
C\ve^{\beta}\|\eta_1\|_{H^{1/2}(\p\widetilde
D)}\|\eta_2\|_{H^{1/2}(\p\widetilde D)} \ , \end{equation}
for any Dirichlet data
$\eta_i\in
C^{1,\alpha}(\p\widetilde D)$, $i=1,2$, and being $C^{1,\alpha}(\p\widetilde D)$ dense into $H^{1/2}(\p\widetilde D)$, the proof of \eqref{est-main} and of Theorem \ref{sub-main} is complete.


\begin{thebibliography}{99}
\bibitem{asing} G. Alessandrini, Singular solutions of elliptic equations and the determination of conductivity by boundary measurements, J. Differential Equations 84 (2) (1990), 252-272.


\bibitem{A} G. Alessandrini, Determining conductivity by boundary
measurements, the stability issue. In: \emph{Applied and Industrial
Mathematics}, R. Spigler(ed.), Kluwer, 1991, 317-324.

\bibitem{ADC}
G. Alessandrini, M. Di Cristo, Stable determination of an inclusion by boundary measurements, SIAM J. Math. Anal. 37 (1) (2005) 200-217.

\bibitem{IP}
G. Alessandrini, L. Rondi, E. Rosset, S. Vessella, The stability for the Cauchy problem for elliptic equations, Inverse Problems 25 123004 (2009), (47pp).


\bibitem{AV} G. Alessandrini and S. Vessella, Lipschitz stability
for the inverse conductivity problem, Adv. Appl. Math., Vol. 35
(2005), 207-241.

\bibitem{AU} H. Ammari and G. Uhlmann, Reconstruction
of the potential from partial Cauchy data for
Schr\"odinger equation, Indiana Univ. Math. J. 53
(2004), no.1, 169-183.

\bibitem{BU} A. Bukhgeim, G.Uhlmann, Recovering a potential from partial Cauchy
data, Comm. Partial Differential Equations, 27 (2002), 653-668.

\bibitem{DiC} M. Di Cristo, Stable determination of an inhomogeneous
inclusion by local boundary measurements, J. Comp. Appl. Math. 198
(2007), no. 2 , 414-425.

\bibitem{F} I. K. Fathallah, Stability for the inverse potential
problem by the local Dirichlet-to-Neumann map for the
Schr\"{o}dinger equation, Appl. Anal. 86 (2007), no. 7,
899-914.

\bibitem{GT}
D.~Gilbarg and N.~S. Trudinger,
\newblock {\em Elliptic partial differential equations of second order}, volume
  224 of {\em Grundlehren der Mathematischen Wissenschaften}.
\newblock Springer-Verlag, Berlin, second edition, 1983.

\bibitem{HW1}
H. Heck, J.-N. Wang, Stability estimates for the inverse boundary value problem by partial Cauchy data, Inverse Problems, Vol 22 (2006), 1787-1796.

\bibitem{HW2}
H. Heck, J.-N. Wang, Optimal stability estimate of the inverse boundary value problem by partial measurements, http://arxiv.org/abs/0708.3289 .

\bibitem{KSU}
C. Kenig, J. Sj\"ostrand, G. Uhlmann,
The Calder\'on problem with partial data,
Ann. of Math. (2) 165 (2007), no. 2, 567--591.

\bibitem{IUY}
O. Imanuvilov, G.  Uhlmann, M. Yamamoto,
The Calder\'on problem with partial data in two dimensions,
Journal American Math. Society, 23(2010), 655-691.

\bibitem{I}
V. Isakov, On uniqueness in the inverse conductivity problem with local
data, Inverse Problems and Imaging, 1 (2007), 95–105.

\bibitem{LCU} M. Lassas, M. Cheney and G. Uhlmann, Uniqueness for a
wave propagation inverse problem in a half space, Inverse Problem
14, no. 3 (1998), 679-684.

\bibitem{LSW} W. Littman, G. Stampacchia and H. F. Weinberger,
Regular points for elliptic equations with discontinuous
coefficients, Ann. Scuola Norm. Sup. Pisa (3) 17 (1963) 43-77.


\bibitem{M}
N. Mandache, Exponential instability in an inverse problem for the
Schr\"odinger equation, Inverse Problems 17 (2001), no. 5, 1435-1444.



\end{thebibliography}
\end{document}